\documentclass[11pt]{article}
\usepackage{mathrsfs}
\usepackage{amssymb,latexsym,tikz,graphicx,eepic,caption}
\usepackage{amsthm}
\usepackage{amsmath,mathrsfs,amssymb}
\usepackage{amscd,mathtools}
\usepackage{rotating}
\usepackage{amsfonts}
\usepackage{graphicx}
\usepackage[all]{xy}
\usepackage[english] {babel}

\usepackage{mathtools}

\usetikzlibrary{arrows,shapes}
\usepackage[section]{placeins}
\usepackage[top=0.5in,bottom=1in,left=1.25in,right=1.25in]{geometry}
\usepackage{graphicx}
\usepackage{verbatim}

\usepackage{setspace}
\usepackage{titlesec}
\usepackage{stmaryrd}

\usepackage{float}
\usepackage{tabularx}

\newtheorem{theorem}{Theorem}[section]
\newtheorem{definition}[theorem]{Definition}
\newtheorem{remark}[theorem]{Remark}
\newtheorem{proposition}[theorem]{Proposition}
\newtheorem{corollary}[theorem]{Corollary}
\newtheorem{example}[theorem]{Example}
\newtheorem{lemma}[theorem]{Lemma}

\newtheorem{theoremintr}{Theorem}

\numberwithin{equation}{section}

\def\Z{\mathbb{Z}}

\begin{document}

	\title{On the multiplicity of the eigenvalues of discrete tori
		\footnote{This work is supported by National Natural Science Foundation of China (Grant 12071371).}
	}
	\author{Bing Xie\footnote{B. Xie,
			Email: {\tt xiebing@sdu.edu.cn}},
		Yigeng Zhao\footnote{Y. G. Zhao,
			Email: {\tt zhaoyigeng@westlake.edu.cn   }}
		and Yongqiang Zhao\footnote{Y. Q. Zhao,
			Email: {\tt yqzhao@wias.org.cn  }}.
	}
	\date{\today}

	\maketitle

	\begin{abstract}
		It is well known that the standard flat torus $\mathbb{T}^2=\mathbb{R}^2/\Z^2$ has arbitrarily large Laplacian-eigenvalue multiplicities. We prove, however, 
		that $24$ is  the optimal upper bound for  the multiplicities of the nonzero eigenvalues of a $2$-dimensional discrete torus. And for general higher dimension discrete tori, we characterize the eigenvalues with large multiplicities. As consequences, we get uniform boundedness results of the multiplicity for a long range and an optimal global bound for the multiplicity. Our main tool of proof is the theory of vanishing sums of roots of unity. 
		
		\vspace{4mm}
		
	\end{abstract}

	\section{Introduction}
	
	 It is well known that the standard flat torus $\mathbb{T}^2=\mathbb{R}^2/\Z^2$ has arbitrarily large Laplacian-eigenvalue multiplicities. More precisely, for the Laplacian eigenvalue equation on the torus given by
	\begin{equation}\label{Laplacian}
		-\Delta u=-\left(\frac{\partial^2}{\partial x^2}+\frac{\partial^2}{\partial y^2}\right)u=\lambda u,\;{\rm on}\;\mathbb{T}^2,
	\end{equation}
	its eigenvalues are $\lambda=4\pi^2(m^2+n^2)$ with corresponding  eigenfunctions $e^{2\pi i(mx+ny)}$,  $ \forall (m,n)\in\mathbb{Z}^2$. The multiplicity of the eigenvalue $4\pi^2 M$ can be explicitly given by 
	\begin{equation*}
		m(4\pi^2 M)=r_2(M) :=\sharp\{(m,n)\in\mathbb{Z}^2|\,M=m^2+n^2 \}.
	\end{equation*}
	The arithmetic function $r_2(M)$ has been well studied in number theory \cite[\S 16.9]{Hardy}.  If the integer $M$ decomposes as  
	$M=2^t\prod\limits_{p\equiv 1 \bmod 4}p^r\prod\limits_{q\equiv 3\bmod 4}q^s$,  then we have 
	
	\[m(4\pi^2 M)=4 \prod_{p \equiv 1\bmod 4} (r+1) \prod_{q\equiv 3\bmod 4} \frac{(-1)^s+1}{2}.\]
	Then it is clear that 
	$$
	\limsup_{M\rightarrow \infty}m(4\pi^2 M)=\infty.
	$$
	
	As discrete analogues and approximations of $\mathbb{T}^2$, it is natural to ask that whether the similar phenomenon occurs for the standard discrete torus $T^2_N=\Z/N\Z \times \Z/N\Z $ for any integer $N\geq 3$.  It is the Cayley graph associated to the group $\Z /N\Z\times \Z/N\Z $ and the generating set $S=\{(0, 1),\,(0, -1),\,(1,0),\,(-1, 0)\}.$
	
	For any $f\in L^2(T^2_N) $, the discrete Laplacian is given by 
	\[ \Delta f(v)=-\sum_{w\in S}(f(v+w)-f(v)).\]
	Recall the basic fact from graph theory that the Laplacian-eigenvalues of $T^2_N$ are 
	\[ \lambda(k_1, k_2)=4\sin^2 \frac{k_1 \pi}{N} + 4\sin^2 \frac{k_2 \pi}{N}, \quad \quad  0\leq k_1, k_2 <N.\]
	Since $T^2_N$ is regular, each Laplacian eigenvalue is one to one corresponding to an eigenvalue of its adjacency matrix, i.e.,  
	\[\mu(k_1, k_2)=2\cos \frac{2k_1 \pi}{N} + 2\cos \frac{2k_2 \pi}{N},  \quad \quad  0\leq k_1, k_2 <N,\]
	and they share the same multiplicity.  Therefore, we may only consider the eigenvalues of the adjacency matrix.

	The discrete torus $T^2_N$ is the product of two copies of the Cayley graph $C_N$ associated to the cyclic group $\Z/N\Z$ with the generating set $\{1,-1\}$.  Recall that, the spectrum of $C_N$ is 
	$$
	{ \rm{Spec}} (C_N)= \begin{bmatrix}
		2& 2\cos\frac{2}{N}\pi & 2\cos\frac{4}{N} \pi& \dots &  2\cos\frac{N-1}{N}\pi \\
		1& 2&2  & \dots &  2
	\end{bmatrix}
	$$ 
	if $N$ is odd, and
	$$
	{\rm{Spec}} (C_N)= \begin{bmatrix}
		2& 2\cos\frac{2}{N}\pi & 2\cos\frac{4}{N}\pi & \dots &  2\cos\frac{N-2}{N}\pi &-2 \\
		1& 2&2  & \dots &  2  &1
	\end{bmatrix}
	$$ 
	if $N$ is even, where the second lines in the above charts denote the multiplicity of the corresponding eigenvalues in the first lines.
	
	By symmetry, each eigenvalue $\mu(k_1,k_2)$, for $0< k_1\neq k_2 <\lfloor \frac{N}{2}\rfloor$, 
	has multiplicity at least $8$, and the multiplicity of  $2\cos\frac{2k\pi}{N}\pm 2$ is at least $4$. But certain eigenvalues may have large multiplicities. For example, in the case that $N=12$, we have the following identities
	\[ 2\cos\frac{8}{12}\pi+2\cos 0=1= 2\cos\frac{4}{12}\pi+2\cos \frac{6}{12}\pi .\]
	This implies the multiplicity of the eigenvalue $1$ is at least $12$.  Therefore, to determine the multiplicities, we have to study identities of cosine functions, i.e., the vanishing sums of cosines.

	For an even positive integer $N\geq 3$, it is easy  to show that  $0=2+(-2)$ is an eigenvalue of $T_N^2$ with multiplicity $2N-2$ (cf. Lemma \ref{mul-0-n=2}). One of  our main results is the following theorem:
	\begin{theoremintr}[Proposition \ref{rough_bound_n=2}, Uniform boundedness of the multiplicity for $d=2$]\label{bound24}
		For any integer $N\geq 3$,  the multiplicity of every non-zero eigenvalue of the adjacent matrix of $T^2 _N$ is uniformly bounded by $24$, and this bound is optimal.
	\end{theoremintr}

	To put Theorem \ref{bound24} into context, we recall the spectral zeta functions of $\mathbb{T}^2=\mathbb{R}^2/\Z^2 $ and $T^2_N $ are given 
	by \[\zeta_{\mathbb{T}^2} (s)= \sum^{\infty}_{i=1} \frac{1}{\lambda^s_i} \quad\quad  \text{and} \quad\quad \zeta_{T^2_N}(s)=\sum^{N^2 -1}_{j=1} \frac{1}{\lambda^s_j},\]
	respectively, where $\lambda_i$ runs through all nonzero Laplacian-eigenvalues of $\mathbb{T}^2$ and $\lambda_j$ through nonzero Laplacian-eigenvalues of
	$T^2_N$. 	A theorem of Chinta, Jorgenson and Karlsson shows that:
	
	\begin{theoremintr}{(\cite{CJK}, \cite[page 586 ]{FK})}\label{CJK}
		For $Re(s)>1$,  one has  $$\lim_{N\to \infty} N^{-2s}\zeta_{T^2_N}(s)= \zeta_{\mathbb{T}^2} (s).$$		
	\end{theoremintr}	
	
	Theorem \ref{CJK} says that the spectral zeta function $\zeta_{\mathbb{T}^2} (s)$ is indeed the limit (up to the scaling factor)
	of the spectral graph zeta functions $\zeta_{T^2_N}(s)$ as $N \to \infty$. Comparing Theorem \ref{CJK} with Theorem \ref{bound24} makes the latter look even more interesting and, in some sense, surprising. 	
	
	Generally speaking,  we know very little about eigenvalue-multiplicity of the Laplacian operator either for Riemannian manifolds or for the discrete cases (graphs).  But it is a natural and basic problem with many important applications. For example, the recent paper of Jiang-Tidor-Yao-Zhang-Zhao \cite{JTYZZ} solves a longstanding problem on equiangular lines in Euclidean spaces by showing an upper bound of the multiplicity of the second eigenvalue of a connected graph. In this paper, we aim to investigate the multiplicity problem for a basic class of examples i.e. the discrete tori, systematically.

	For a general discrete torus $T_N^d$ with $d \in \mathbb{N}$, the Cayley graph associated to the group $(\Z/N\Z)^d$ and the generating set $ \{\pm e_1, \dots, \pm e_d\}$, with $e_i$ the canonical basis with $1$ at the $i$-th coordinate and $0$ at all the other coordinates,  
eigenvalues of its adjacency matrix are 
	\[\mu(k_1, k_2, \dots, k_d)=2\cos \frac{2k_1 \pi}{N} + 2\cos \frac{2k_2 \pi}{N}+ \cdots + 2\cos \frac{2k_d \pi}{N},  \quad  0\leq k_1, k_2, \dots, k_d <N.\]	
	
	The multiplicities of eigenvalues problems are much more involved and complicated. For example, if $2|N$ then $0$ is an eigenvalue of $T^2_N$. Therefore, $2\cos\frac{2k\pi}{N}=2\cos\frac{2k\pi}{N}+0$ is an eigenvalue of $T^3_N$ with multiplicity at least $2N-2$, for any $0\leq k \leq N$. This implies that we have many eigenvalues with large multiplicities if $d\geq 3$. Given this phenomenon, however, we still have the following uniformity result.
	
	\begin{theoremintr}[Corollary \ref{rang_bound}, Uniform boundedness of the multiplicity for higher dimensional tori]\label{bound-d} For $N\geq 3$, we list all eigenvalues of $T^d_N$ as $\mu_1\geq \mu_2 \geq \cdots \geq \mu_{N^d}$. There exists a constant $c=c_d$, $0< c \leq \frac{1}{2}$, such that if $k\in [1, cN^d) \cup ((1-c)N^d, N^d] $,  then the multiplicity $m_{T^d_N}(\mu_k)$ is uniformly bounded by a constant $C_d$ which only depends on $d$. 	
	\end{theoremintr}	
	
	Theorem \ref{bound-d}  implies that the multiplicity of eigenvalues of $ T^d_N$ is uniformly bounded for a long range. Comparing Theorem \ref{bound24} with Theorem \ref{bound-d}, it is clear that by taking the constant  $c=\frac{1}{2}$ in the latter one gets Theorem \ref{bound24} back (without giving the optimal bound $24$). We are, however, unable to give sharp constants $C_d$ in Theorem \ref{bound-d} as the one in Theorem \ref{bound24}. Moreover, the constant $c_d$ is much weaker for general $d\geq 3$ and we do not seek to optimize this constant. 
	
	For general $T^d_N$, we also get global bounds for the multiplicities with the optimal exponent $\frac{d}{2}$ (see Remark \ref{optimal_bound_all} for an example with the optimal exponent).

	\begin{theoremintr}[Corollary \ref{general_bound}]\label{total bound}
		Let $\mu$ be an eigenvalue of $T^d_N$ and $p_1$ the minimal prime factor of $N$.  Then $m_{T^d_N}(\mu)\leq C_dN^{\frac{d}{p_1}} $, where $C_d$ is a constant, which depends only on $d$. In particular, $m_{T^d_N}(\mu)\leq C_dN^{\frac{d}{2}}$.
	\end{theoremintr}

	The main ingredient of the proofs of the above Theorems is the theory of vanishing sums of  cosines or roots of unity with a given length. For Theorem \ref{bound24}, the classification of vanishing sums of cosines of length $4$ due to Wlodarski \cite{Wlodarski} or those of roots of unity of length $8$ due to Conway-Jones \cite{Conway} is sufficient. For the classification of vanishing sums of roots of unity, the longer length the sum is, the more complicated and more difficult the classification is. Difficulties already arise when we classify of vanishing sums of length, say, $16$. See \cite{Poonen, Christie}.  A classification of vanishing sum of roots of unity of arbitrary length seems, at least to the authors, not practical, if not at all impossible.  Hence for higher dimensional torus, we seek to characterize those eigenvalues with large multiplicities instead.  
	
	For $T^2_N$, we see the exceptional large multiplicity of the zero eigenvalue. For the higher dimensional case, we first consider the multiplicity of the zero eigenvalue and give explicit criteria for zero being an eigenvalue and having large multiplicity in Proposition \ref{zero_eigenvalue_critierion} and  Proposition \ref{zero_infinity_iff}, respectively. Then we obtain a criterion of large multiplicity for a general eigenvalue $\mu \in \sigma(T^d_N)$. Roughly speaking, the large multiplicity of $\mu$ is due to the facts that $\mu\in \sigma(T^{d-r}_N)$ and $0\in \sigma(T^r_N)$ with a large multiplicity, for some $1 \leq r < d$.  The results are summarized as follows: 
	
\begin{theoremintr}[Theorem \ref{all_infinity_iff}]\label{mainprop}
  Let  $N=p_1^{a_1}p_2^{a_2}\cdots p^{a_s}_s$ with $p_1< p_2<\dots<p_s$ are distinct primes. Let $d\geq 2$ and $\mu$ be an eigenvalue of $T^d_N$. Then the following are equivalent:
  \begin{itemize}
   \item[(i)] there exists $r\in \mathbb{Z}_{\geq 1}, b_1,\cdots, b_s \in \Z_{\geq 0}$  such that $\mu$ is an eigenvalue of $T^{d-r}_N$ and $2r=\sum\limits_{l=1}^{s}b_lp_l $ with at least one of $b_l\geq 2$;
   \item[(ii)]  the multiplicity $m_{T^d_N}(\mu)$ of $\mu$ satisfies $m_{T^d_N}(\mu)\geq c_dN$ for some explicit constant $c_d>0$ which depends only on $d$. 
  \end{itemize}
 \end{theoremintr}
	
	The above  theorem shows a special gap phenomenon that the multiplicity is either uniformly bounded or of the order at least $cN$,  for some constant $c$ independent of $N$. Both Theorem \ref{bound-d} and Theorem \ref{total bound} are consequences of Theorem \ref{mainprop}.

	The paper is organized as follows. In Section 2 we prove the uniform bound for $T^2_N$, then we recall some standard results on vanishing  sums of roots of unity in Section 3. The criterion of zero being an eigenvalue with large multiplicity are studied in Section 4, and we prove the main result Theorem \ref{mainprop} on the multiplicities of any eigenvalues in Section 5. In the last section,  we briefly discuss the eigenvalue multiplicity problem for general abelian Cayley graphs.   
	
	\section{Multiplicities of eigenvalues for $T^2_N$}
	
	In this section, we present a uniform bound for the multiplicities of  non-zero eigenvalues  for the Cayley graph $T^2_N=\Z /N\Z\times \Z/N\Z $ with  generating set  $\{(0, 1),\,(0, -1),\,(1,0),\,(-1, 0)\}$.
	
	For an even integer $N\geq 3$, both $2$ and $-2$ are eigenvalues of $C_N$, hence $0=2+(-2)$ is an eigenvalue of $T^2_N$.  Its multiplicity is explicitly given in the following lemma:
	\begin{lemma}\label{mul-0-n=2}
		For an even integer $N$, the multiplicity of $0$ in $T^2_N$ is $2N-2$.
	\end{lemma}
	\begin{proof}
		It is enough to determine the solution of the following equation:
		\begin{equation*}
			2\cos\frac{2k_1}{N}\pi+ 2\cos\frac{2k_2}{N}\pi =0,
		\end{equation*}
		for $k_1,k_2\in \Z/N\Z$. Therefore, we have 
		\[ k_1\pm k_2 \equiv \frac{N}{2} \mod N.\]
		Choosing a representative class of $\Z/N\Z$, say $\{0,1,2,\cdots, N-1\}$, we may list the following four possibilities:
		\begin{itemize}
			\item[(i)] the solutions of $k_1+k_2 = \frac{N}{2} $: 
			$(k_1,k_2)=(i,\frac{N}{2}-i)$, for $0\leq i \leq \frac{N}{2}$;
			\item[(ii)] the solutions of $k_1+k_2 = \frac{3N}{2} $: 
			$(k_1,k_2)=(\frac{N}{2}+i,N-i)$, for $1\leq i \leq \frac{N}{2}-1$;
			\item[(iii)] the solutions of $k_1-k_2 = \frac{N}{2} $: 
			$(k_1,k_2)=(\frac{N}{2}+i,i)$, for $1 \leq i \leq \frac{N}{2}-1$;
			\item [(iv)] the solutions of $k_1-k_2 = -\frac{N}{2} $: 
			$(k_1,k_2)=(i,\frac{N}{2}+i)$, for $1 \leq i \leq \frac{N}{2}-1$.
		\end{itemize}
		There are $2N-2$ distinct solutions in total. Hence the assertion follows.
	\end{proof}
	
	\begin{definition}\label{Eigenvalue}
		For $k_1,\,k_2\in \Z/N\Z$, we denote the eigenvalue as
		$$
		\mu(k_1,k_2):=\mu_N(k_1,k_2):= 2\cos\frac{2k_1}{N}\pi+ 2\cos\frac{2k_2}{N}\pi,
		$$
		and its multiplicity as
		$$
		m(k_1,k_2)=	m_{T^2_N}(k_1,k_2):=\sharp \{\mu\in  \sigma({ T_N^2}) |\mu=\mu(k_1,k_2) \} .
		$$
	\end{definition}
	It is clear that $\mu_N(k_1,k_2)=\mu_N(k_2,k_1)$ and $\mu_N(k_1,k_2)=\mu_{dN}(dk_1,dk_2)$, for any $d\in \mathbb{N}$.  Thus we have $m_{dN}(dk_1,dk_2)\geq m_{N}(k_1,k_2)$ for any $d\in \mathbb{N}$.
	
	For an even positive  integer $N$, it is clear that
	$$
	\mu(N/2-k_1,N/2-k_2)=-\mu(k_1,k_2).
	$$

	As we have seen that the multiplicities closely relate to the vanishing sums of cosines.
	Thanks to the following result of Wlodarski, there are only finitely many types of vanishing sums of length $\leq 4$.
	\begin{theorem}{(\cite{Wlodarski})} \label{Woldarski}
		If $ \alpha_j$ commensurable with $\pi$ (i.e. $\alpha_j\in \mathbb{Q}\pi$) satisfy $0 \leq \alpha_j \leq \pi,\; j=1,2,3,4$. Then all the solutions $\{\alpha_1, \,\alpha_2,\,\alpha_3, \,\alpha_4\}$ of the equation
		\begin{equation}\label{vanishing_sum}
			\sum_{j=1}^4 \cos\alpha_j =0, 
		\end{equation}
		belong to the following quadruples
		\begin{flalign*}
			&(i) \;\{\alpha_1,\,\alpha_2,\,\pi-\alpha_1,\,\pi-\alpha_2\},\;0\leq\alpha_1\leq\alpha_2\leq \frac{\pi}{2};
			&	(ii)  \;\{\delta,\,\frac{2\pi}{3}-\delta,\,\frac{2\pi}{3}+\delta,\frac{\pi}{2}\},\;0\leq\delta\leq \frac{\pi}{3};\\
			&(iii)\;\{\frac{2\pi}{5},\,\frac{4\pi}{5},\,\frac{\pi}{2},\,\frac{\pi}{3}\},\; \{\frac{3\pi}{5},\,\frac{\pi}{5},\,\frac{\pi}{2},\,\frac{2\pi}{3}\};
			&	(iv)  \;\{\frac{\pi}{5},\,\frac{3}{5}\pi,\,\frac{\pi}{3},\,\pi\},\; \{\frac{4}{5}\pi,\,\frac{2}{5}\pi,\,\frac{2}{3}\pi,\, 0\}; \\
			&	(v) \;\{\frac{2\pi}{5},\,\frac{7\pi}{15},\,\frac{13\pi}{15},\,\frac{\pi}{3}\},\; \{\frac{3\pi}{5},\,\frac{8\pi}{15},\,\frac{2\pi}{15},\,\frac{2\pi}{3}\};
			&	(vi) \;\{\frac{\pi}{15},\,\frac{11\pi}{15},\,\frac{4\pi}{5},\,\frac{\pi}{3}\},\; \{\frac{14\pi}{15},\,\frac{4\pi}{15},\,\frac{\pi}{5},\,\frac{2\pi}{3}\};\\
			&	(vii) \;\{\frac{2\pi}{7},\,\frac{4\pi}{7},\,\frac{6\pi}{7},\,\frac{\pi}{3}\},\; \{\frac{5\pi}{7},\,\frac{3\pi}{7},\,\frac{\pi}{7},\,\frac{2\pi}{3}\}.
		\end{flalign*}
	\end{theorem}
	\begin{remark}
		The above theorem of Wlodarski is obtained via case-by-case analysis and a little bit of Galois theory of $(\Z/N\Z)^*$ acting on $N$-th roots of unity. More precisely, if
		$\sum\limits_{j=1}^4 \cos \alpha_j=0$, then, for any integer $r$ coprime to $N$, $\sum\limits_{j=1}^4 \cos r\alpha_j=0$.
	\end{remark}

	A direct consequence of Wlodarski's result is the following proposition:
	\begin{proposition}\label{general_multiplicities}
		\begin{itemize}
			\item[(i)] 	For an odd integer $N$, we have
			\begin{flalign*}
				& m(0,0)=1; \;m(0,k)=m(k,k)=4,\;{\rm for}\; k\in \Z/N\Z-\{0\};\\
				& m(k_1,k_2)=8,\;{\rm for}\; k_1\not= k_2\in \Z/N\Z-\{0\}.
			\end{flalign*}
			\item[(ii)] For an even integer $N$, with $12\nmid N$, $30\nmid N$ and $42\nmid N$, we have
			\begin{flalign*}
				& m(0,N/2)=2N-2; \;m(0,0)=1=m(N/2,N/2);\\
				& m(0,k)=m(N/2,k)=m(k,k)=4,\;{\rm for}\; k\in \Z/N\Z-\{0,\frac{N}{2}\};\\
				& m(k_1,k_2)=8,\;{\rm for}\;k_1\not= k_2\in \Z/N\Z-\{0,\frac{N}{2}\}.
			\end{flalign*}
		\end{itemize}
		
	\end{proposition}
	\begin{proof}
		In the case that $k_1k_2=0$ or $k_1=k_2$, the multiplicities could be determined easily.  For even $N\geq 4$, we have shown that $m(0,N/2)=2N-2$ in Lemma \ref{mul-0-n=2}. 
		By the symmetry, we have 
		$$m(k_1,k_2)\geq 8 \;\mathrm{if}\;\begin{cases}
			2\nmid N,\; k_1\neq k_2 \in \Z/N\Z\setminus\{0\};\\
			2\mid N,\; k_1\neq k_2 \in \Z/N\Z\setminus\{0,\frac{N}{2}\}.
		\end{cases}$$
		
		Note that, for odd $N$ or even $N$ with $12\nmid N, 30\nmid N$ and $42\nmid N$,  there are no nontrivial equations of the following form:
		$$
		2\cos\frac{2k_1\pi}{N}+ 2\cos\frac{2k_2\pi}{N}  =2\cos\frac{2k_3\pi}{N}+ 2\cos\frac{2k_4\pi}{N},
		$$
		or equivalently, no non-trivial solutions for
		$$
		\cos\frac{2k_1\pi}{N}+ \cos\frac{2k_2\pi}{N} +\cos (\pi-\frac{2k_3\pi}{N})+ \cos(\pi-\frac{2k_4\pi}{N})=0,
		$$
		by  Theorem  \ref{Woldarski}.	Therefore all the multiplicities are given by trivial relations and symmetry.
		
	\end{proof}
	
	To conclude our main Theorem \ref{bound24}, we treat the special cases that $12\mid N$, or $30\mid N$ or $42\mid N$ together instead of giving an explicit bound for each eigenvalue. 
	\begin{proposition}\label{rough_bound_n=2}
		For any integer $N\geq 3$,  we have an optimal bound $$m_{T^2_N}(k_1,k_2)\leq 24,$$ for any $k_1,k_2\in \Z/N\Z$ with $\mu_N(k_1,k_2)\neq 0$. 
	\end{proposition}
	\begin{proof}
		We only need to consider the case that $N$ is even by Proposition \ref{general_multiplicities}. 
		Given $\mu(k_1,k_2)=2\cos\frac{2k_1}{N}\pi+ 2\cos\frac{2k_2}{N}\pi \neq 0$ and $k_1\neq k_2$, we know that, for any $k'$ and $k''$ with $\{k_1,k_2\}\cap \{k',k''\}=\emptyset$, a relation
		\begin{align}\label{eq:eigen_eq_2}
			2\cos\frac{2k_1\pi}{N}+ 2\cos\frac{2k_2\pi}{N}  =2\cos\frac{2k'\pi}{N}+ 2\cos\frac{2k''\pi}{N} \end{align}
		will add the multiplicity of $\mu(k_1,k_2)$ at most $8$.
		This equation \eqref{eq:eigen_eq_2} can be reformulated as 
		\begin{align}\label{eq:VS_cos_4}
			\begin{cases}
				2\cos\frac{2k_1}{N}\pi+ 2\cos\frac{2k_2}{N}\pi +2\cos\frac{2k_3}{N}\pi+ 2\cos\frac{2k_4}{N}\pi=0\\
				2\cos\frac{2k_1}{N}\pi+ 2\cos\frac{2k_2}{N}\pi \neq 0\\
				2\cos\frac{2k_1}{N}\pi+ 2\cos\frac{2k_3}{N}\pi \neq 0.
			\end{cases}
		\end{align}
		Note that $\frac{2k_3}{N}\pi=\pi-\frac{2k'}{N}\pi$, the assumption $\{k_1,k_2\}\cap \{k',k''\}=\emptyset$ implies the last inequality above.	In fact the last two inequalities in \eqref{eq:VS_cos_4} are equivalent to say that  any sum of two out of the four terms in the first equation is non-zero.  
		
		Now, for a given $(k_1,k_2)$, we  search for $(k_3,k_4)$ with $0\leq \frac{2k_3}{N}\pi,\frac{2k_4}{N}\pi\leq \pi$ satisfying \eqref{eq:VS_cos_4}. Therefore, by Theorem \ref{Woldarski} of Wlodarski, 
		we  have at most 2 choices of $(k_3, k_4)$ which satisfies \eqref{eq:VS_cos_4}. Therefore 
		\[   m_{T^2_N}(k_1,k_2)\leq 8+2\times 8=24.\]
		
		For example, if $(\frac{2k_1}{N}\pi,\frac{2k_2}{N}\pi)=(\frac{4}{5}\pi, \frac{1}{3}\pi)$, the solutions of $(\frac{2k_3}{N}\pi,\frac{2k_4}{N}\pi)$ are $(\frac{2}{5}\pi,\frac{1}{2}\pi)$ and $(\frac{1}{15}\pi,\frac{11}{15}\pi)$.  Hence in the case that $N=60$, we have 
		\[m_{T^2_{60}}(24,10)=24.\]
		This implies the bound $24$ is optimal.
	\end{proof}
	
	\begin{example}
		Here we list all higher multiplicities (bigger than $8$) of eigenvalues of $T^2_{60}$ in the following table:
		\begin{center}

			\begin{tabularx}{135mm}{c|X}\hline
				Multiplicities & Eigenvalues  of $T^2_{60}$ \\ \hline 
				12 & $ \pm(2\cos\frac{\pi}{5}+1), \pm(2\cos\frac{2}{5}\pi-1)$\\ \hline
				16&   $\pm(2\cos\frac{\pi}{15}+1)$\\ \hline
				20 &  $\pm 1$ \\ \hline
				24 &  $\pm 2\cos\frac{\pi}{5},\pm 2\cos\frac{2}{5}\pi$ \\ \hline
				118 & 0\\ \hline
			\end{tabularx}
			\end{center}
				
	\end{example}

	\section{Vanishing sums of cosines}

	For any integer $N\geq 3$, let   $T^d_N=C_N\times\cdots \times C_N$ be the $d$-dimensional 
	discrete torus. It is a Cayley graph associated to the group $(\Z/N\Z)^d$ with the generating set 
	$ \{\pm e_1, \dots, \pm e_d\}$, where $e_i$ is the canonical element with $1$ at the $i$-th coordinate and $0$ at all other coordinates.
	
	The eigenvalues of $T^d_N$ are the sum 
	\[ 2\sum_{j=1}^{d}\cos\frac{2k_j\pi}{N},\]
	for $0\leq k_j\leq \lfloor \frac{N}{2}\rfloor, \forall j$.  A natural question is to calculate the multiplicities of these eigenvalues.
	In general, the multiplicity may increase once there is an equality
	\[ 2\sum_{j=1}^{d}\cos\frac{2k_j\pi}{N}= 2\sum_{j=1}^{d}\cos\frac{2k'_j\pi}{N},\]
	for $k\neq k' \in \Z^d$. Note that $2\cos(\alpha)=e^{i\alpha}+e^{-i\alpha}$ for any $\alpha\in \mathbb{R}$, we may rewrite this equality as a vanishing sum of roots of unity
	\[ \sum_{j=1}^{d}(e^{\pi i\frac{2k_j}{N}}+e^{-\pi i\frac{2k_j}{N}}+e^{\pi i\frac{N-2k'_j}{N}}+ e^{-\pi i\frac{N-2k'_j}{N}})= 0.\]
	Roughly speaking, the multiplicities of eigenvalues are determined by the numbers of vanishing sums of roots of unity. 
		
	Let $U_{\infty}=\{z\in \mathbb{C}| z^n=1 \;\text{for some } n\in \mathbb{Z}_{>0}\}$ be the group of roots of unity, and for any positive integer $N$, $U_{N}=\{z\in \mathbb{C}| z^N=1 \}$ be the subgroup of $N$-th roots of unity. Let $\mathbb{C}[U_\infty]$ be the group ring and $v:  \mathbb{C}[U_\infty]\to \mathbb{C}$ be the nature map by taking the value of  a formal sum.
	\begin{definition}
		\begin{itemize}
			\item[(i)] A vanishing sum of roots of unity is a formal sum $$M=a_1\zeta_1+a_2\zeta_2+\cdots+a_k\zeta_k\in \mathbb{C}[U_\infty]$$ with $v(M)=0$. We also write $a_1\zeta_1+a_2\zeta_2+\cdots+a_k\zeta_k=0$ for short.
			\item[(ii)] For a vanishing sum $M=a_1\zeta_1+a_2\zeta_2+\cdots+a_k\zeta_k\in \mathbb{N}[U_\infty]$, let $\varepsilon_0(M)$ be the cardinality of the set $\{j|a_j\neq 0\}$ and the length $\varepsilon(M)$  of $M$ is the sum of such nonzero coefficients.
			\item[(iii)]  A vanishing sum of roots of unity $M=\sum\limits_{j=1}^{k}a_j\zeta_j=0$ with the coefficients $a_j\in \mathbb{Z}_{>0}$ is called minimal if 
			\[ \sum_{j=1}^kb_k\zeta_k=0, \; a_j\geq b_j\in \mathbb{N}\]
			implies either $a_j=b_j$ for all $j$ or $b_j=0$ for all $j$. In other words, $M$ is minimal if it cannot be written as two non-trivial (i.e., all coefficients are zero) vanishing sums of roots of unity.
		\end{itemize}
	\end{definition}
	\begin{example}\label{ex_symmetric}
		Let $p$ be an odd prime and $\zeta_p$ be a primitive $p$-th root of unity.  The sum $R_p:=1+\zeta_p+\cdots+\zeta_p^{p-1}$ is a minimal vanishing sum of roots of unity.  So is any rotation $\zeta R_p$ by  $\zeta\in U_\infty$. This kind of vanishing sums is called symmetric. The other minimal vanishing sums of roots of unity in $\mathbb{Z}_{>0}[U_\infty]$ are called asymmetric, for example $(\zeta_3+\zeta_3^2)(\zeta_5+\zeta_5^2+\zeta_5^3+\zeta_5^4)+\zeta_7+\zeta_7^2+\zeta_7^3+\zeta_7^4+\zeta_7^5+\zeta_7^6$ is an asymmetric vanishing sum of roots of unity.
	\end{example}

	\begin{theorem}\label{Lam-Leungthm}
		Let $N=p_1^{a_1}\cdots p^{a_s}_s$ with $p_1<p_2<\dots<p_s$ are distinct primes. 
		\begin{itemize}
			\item[(i)]{(\cite[Theorem 4.8]{LL})} Let $M$ be a minimal vanishing sum in  $\mathbb{Z}_{>0}[U_N]$. Then either $M$ is symmetric or
			$\varepsilon(M)\geq \varepsilon_0(M)\geq (p_1-1)(p_2-1)+p_3-1$.
			\item[(ii)] {(\cite[Theorem 5.2]{LL})} Let 
			\[ W(N)=\{\varepsilon(M)\;|\; M\in \mathbb{N}[U_N] \text{ is a vanishing sum of roots of unity}\}.\]
			Then $W(N)=\sum\limits_{j=1}^{s}\mathbb{N}p_j.$
		\end{itemize}
	\end{theorem}
	
	An immediate consequence of the above theorem is the following:
	\begin{lemma}\label{dim_prime_decomp}
		For $d\geq 1$ and $N=p_1^{a_1}\cdots p^{a_s}_s$ with $p_1,\dots,p_s$ are distinct primes.
		If zero is an eigenvalue of $T^d_N$, then  $2d\in \mathbb{N}p_1+\cdots+\mathbb{N}p_s$.
	\end{lemma}
	
	\begin{proof}
		
		Since any eigenvalue of  $T^d_N$ can be written as $2\sum\limits_{j=1}^{d}\cos\frac{2k_j\pi}{N}$ for some $k_j\in \mathbb{Z}_{\geq 0}$, we may assume that  
		$2\sum\limits_{j=1}^{d}\cos\frac{2k_j\pi}{N}=0$.  Using the identity $2\cos(\alpha)=e^{i\alpha}+e^{-i\alpha}$, we get a vanishing sum of roots of unity
		\[ \sum_{j=1}^{d}(e^{\frac{2ik_j\pi}{N}}+e^{\frac{-2ik_j\pi}{N}})=0.\]
		
		It is a vanishing sum of roots of unity of length $2d$. Therefore the claim follows from Theorem \ref{Lam-Leungthm}.
	\end{proof}	 
	
	Note that if $2\mid N$, then the condition 	 $2d\in \mathbb{N}p_1+\cdots+\mathbb{N}p_s$ is automatically satisfied. Therefore this condition is only for odd  $N$. In the next section we give an explicit criterion for zero being an eigenvalue.
	
	To give a uniform bound of multiplicities, we also need the following Theorem of Everste.
	\begin{theorem}{(\cite[Theorem]{Evertse})} \label{Evertsethm}
		Let $ a_1\zeta_1+\cdots+a_k\zeta_k\in \mathbb{C}[U_\infty]$ with $\prod_{j=1}^{n}a_j\neq 0$. Then the equation 
		\[ a_1\zeta_1+\cdots+a_k\zeta_k=1\]
		has at most $(k+1)^{3(k+1)^2}$ non-degenerated solutions $(\zeta_1,\dots,\zeta_k)$. That is, for any proper and non-empty subset $J$ of \{1,\dots,k\}, the solution $(\zeta_1,\dots,\zeta_k)$ satisfies \[\sum_{j\in J}a_j\zeta_j \neq 0.\]
	\end{theorem}
	
	\begin{example}
		We don't know how to introduce the notion of  minimal or non-degenerated for sums of cosines. For example, $2\cos 0+2\cos\frac{2\pi}{3}=1$ is a sum of cosines, which should be thought as ``non-degenerated" as an eigenvalue of $T^2_3$, but it comes from a degenerated sum of roots of unity $1+1+\zeta_3+\zeta_3^2=1$. 
	\end{example}
	
	Any vanishing sum of roots of unity can be written as a sum of minimal vanishing sums, therefore lots of questions on vanishing sums of roots of unity can be reduced to study the minimal ones. However, the above example shows that it is difficult to define minimal or non-degenerated for sums of cosines, so we need find another way to relate them. The following Lemma is an attempt in this direction.
	
	\begin{lemma} \label{conj_formal_sum}
		Let $N$ be an odd positive integer and $e^{i\alpha_{1}},\cdots,e^{i\alpha_{d}}$ be $N$-th roots of unity.
		Assume the vanishing sum 
		\begin{align}
			\sum_{j=1}^{d} (e^{i\alpha_j}+e^{-i\alpha_j})=\sum_{k=1}^{r}\eta_{k}R_{p_k} \in \mathbb{C}[U_N],
		\end{align}
		where $|\eta_{k}|=1$ and  $R_{p_k}=1+\zeta_{p_k}+\cdots+\zeta_{p_k}^{p_k-1}$ where each prime $p_k \mid N$ and $p_1,\cdots, p_r$ are distinct. We further assume that each $\eta_{k}$ is not a $p_k$-th roots of unity if $\eta_{k}\neq 1$. Then we have
		\[ \eta_1=\cdots=\eta_r=1.\] 
	\end{lemma}
	\begin{proof} Note that if $\eta_k=1$, we may then subtract $R_{p_k}$ on both sides and reduce the length of vanishing sums. Hence in the following we assume all
		$\eta_{k}\neq 1$. We deduce a contradiction in following steps:

		(i) For any $0\leq a< p_k$, the conjugate of $\eta_{k}\zeta_{p_k}^a$ does not appear in $\eta_{k}R_{p_k}$. Otherwise, we assume that $\overline{\eta_{k}\zeta_{p_k}^a}=\eta_{k}\zeta_{p_k}^b$. It follows that $\eta_k^{2p_k}=1$, which contradicts to the assumption that $N$ is odd and $\eta_k\neq 1$.

		(ii) For any two integers $a,b$ such that $0\leq a\neq b<p_k$, the conjugates of $\eta_k\zeta_{p_k}^a$ and $\eta_k\zeta_{p_k}^b$ cannot appear in any $\eta_lR_{p_l}$ simultaneously for $l\neq k$.  By taking contradiction, without loss of generality, we assume that  $\overline{\eta_1\zeta_{p_1}^a}= \eta_2\zeta_{p_2}^c$ and $\overline{\eta_1\zeta_{p_1}^b}= \eta_2\zeta_{p_2}^d$  for some $0\leq c\neq d \leq p_2$. This implies $\zeta_{p_1}^{b-a}=\zeta_{p_2}^{c-d}$, which is impossible.

		(iii) Assume $p_{k_0}=\mathrm{max}\{p_1,\dots,p_r\}>r$, consider the conjugates of each $\eta_{k_0}\zeta_{p_{k_0}}^{a_j}$ for $0\leq a_j<p_{k_0}$, there are $p_{k_0}$ elements, but they appear in at most $r-1$  classes $\{\eta_{k}R_{p_k}, k=1,\dots,r, k\not= k_0\}$, which contradicts to (ii) since $p_{k_0}>r$.
	\end{proof}

	\section{Multiplicities of zero eigenvalues of $T^d_N$}
	We have seen that the zero eigenvalue in $T^2_N$ plays a special role in the spectrum. For $\mu\in \sigma(T^d_N)$, it may happen that $\mu$ is also an eigenvalue of a lower dimension discrete torus $T^{d-r}_N$ for some $r\geq 1$. Then $0$ is an eigenvalue of $T^r_N$ in this case. We will show that all large multiplicities are caused by large multiplicities of the zero eigenvalue in lower dimensions in Theorem \ref{all_infinity_iff}. In this section, we treat the zero eigenvalue first and general eigenvalues in the next section.
	\subsection{Criterion for zero eigenvalue}
	For any prime number $p$ and $\delta\in \mathbb{Q}\pi$,  the vanishing of $e^{i\delta}(1+\zeta_p+\zeta_p^2+\dots+\zeta_p^{p-1})$ gives us the following vanishing sum of cosines:
	\begin{align*}
		C_p(\delta):=\cos\delta+\sum_{j=1}^{\frac{p-1}{2}}\big(\cos(\frac{2j\pi}{p}+\delta)+\cos(\frac{2j\pi}{p}-\delta)\big).
	\end{align*}
	The sum $2C_p(\delta)$ is  a vanishing sum with $p$ terms of eigenvalues of $\Z/N\Z$ for some $N\in \mathbb{N}$.

	\begin{proposition}\label{zero_eigenvalue_critierion}
		For $d\geq 1$ and $N=p_1^{a_1}p_2^{a_2}\cdots p^{a_s}_s$ with $p_1< p_2<\dots<p_s$ are distinct primes.
		\begin{itemize}
			\item[(i)] If $2 \nmid N $, then $0$ is an eigenvalue of $T^d_N$ if and only if $2d\in \mathbb{N}p_1+\cdots+\mathbb{N}p_s$.
			\item[(ii)] If $2 \mid N $ and  $d$ is even, then $0$ is always an eigenvalue of $T^d_N$.
			\item[(iii)] If $2\mid N$ and $d$ is odd and if $p_2 \leq d$, then $0$ is an eigenvalue of $T^d_N$.
			\item[(iv)]  If $2\mid N$ and $d$ is odd and if $p_2>d$, then $0$ is an eigenvalue of $T^d_N$ if and only if $4\mid N$.
		\end{itemize}
	\end{proposition}

	\begin{proof} 
		(i) The necessary part is  Lemma \ref{dim_prime_decomp}. We show the sufficient part in a constructive way. Assume $2d=\sum\limits_{l=1}^{s}b_lp_l$. Since $2\nmid N$, it implies that $2\mid \sum\limits_{l=1}^{s}b_l $. For each $p_l$,  we have  the following vanishing sum of cosines of length $p_l$:
		\[C_{p_l}(0)=\cos 0+2\cos\frac{2\pi}{p_{l}}+2\cos\frac{4\pi}{p_{l}}+\cdots+ 2\cos\frac{(p_l-1)\pi}{p_{l}}=0.\]
		Therefore
		\[ 0=(\sum\limits_{l=1}^{s}b_l )\cos0+\sum_{l=1}^{s}b_l(2\cos\frac{2\pi}{p_{l}}+2\cos\frac{4\pi}{p_{l}}+\cdots+ 2\cos\frac{(p_l-1)\pi}{p_{l}}).\]
		Note that there are $2d$ terms on the right hand side of the above equality. We rewrite it as 
		\[ 0=(\frac{1}{2}\sum\limits_{l=1}^{s}b_l )(2\cos0)+\sum_{l=1}^{s}b_l(2\cos\frac{2\pi}{p_{l}}+2\cos\frac{4\pi}{p_{l}}+\cdots+ 2\cos\frac{(p_l-1)\pi}{p_{l}}),\]
		which is a sum of $d$ eigenvalues of $\Z/N\Z$.
		Hence $0$ is an eigenvalue of $T^d_N$.
		
		(ii) Since $d$ is even and $2\mid N$, we have $p_1=2$ and write $d=2b$.  Since both $2\cos0$ and $2\cos\pi$ are eigenvalues of $\Z/N\Z$.  Hence
		\[ 0=b(2\cos0+2\cos\pi)\]
		is a sum of $d$ eigenvalues of $\Z/N\Z$. That is $0$ is an eigenvalue of $T^d_N$.

		(iii) Since $d$ is odd, then $d-p_2=2c$ for some $c\in \mathbb{N}$.  Therefore $2d=(2c)\times 2+ 2p_2$. We formulate 
		\[ 0=c(2\cos0+2\cos\pi)+2\cos0+2(2\cos\frac{2\pi}{p_{2}}+2\cos\frac{4\pi}{p_{2}}+\cdots+ 2\cos\frac{(p_2-1)\pi}{p_{2}}),\]
		which is a sum of $2c+p_2=d$ eigenvalues of $\Z/N\Z$.

		(iv) The sufficient part is easy. Since once we have $4\mid N$, then 0=$2\cos\frac{\pi}{2}=2\cos\frac{2\frac{N}{4}\pi}{N}$ is an eigenvalue of $\Z/N\Z$. Hence  $0$ is also an eigenvalue of $T^d_N$.

		Now we show the necessary part.  Assume that $0=2\sum\limits_{j=1}^{d}\cos\frac{2k_j\pi}{N}$ for some $0\leq k_j <\lfloor \frac{N}{2}\rfloor$. By our assumption, any odd prime $p\leq d$ is coprime to $N$, and it follows that 
		\[ 0=2\sum\limits_{j=1}^{d}\cos\frac{2pk_j\pi}{N}\]
		for any such prime $p$ by the natural Galois action induced by $e^{\frac{2\pi k_i}{N}}\mapsto e^{\frac{2\pi p k_i}{N}} $ on $N$-th roots of unity, and hence, for any integer $2\nmid n\leq d$, 
		\[ 0=2\sum\limits_{j=1}^{d}\cos\frac{2nk_j\pi}{N}.\]
		We write $x_j=\cos\frac{2k_j\pi}{N}$ and $\cos(n\theta)=T_n(\cos\theta)$ the Chebyshev polynomial of the first kind.
		Therefore, we have 
		\begin{equation}\label{linear_system}
			\setlength\arraycolsep{1pt}
			\left\{
			\begin{array}{rr}
				\sum\limits_{j=1}^{d}x_j=0 &  \\
				\sum\limits_{j=1}^{d}T_n(x_j)=0  &  \;\mathrm{for\; any\;} 2\nmid n\leq d.
			\end{array}
			\right.
		\end{equation}
		
		Let $ P_m(x_1,\dots,x_d)=\sum_{j=1}^{d}x_j^m$. By the property of Chevbyshev polynomials (of the first kind) and induction on $n$ with $2\nmid n\leq d$, we get
		\[P_n(x_1,\dots,x_d)=0, \mathrm{for\; any\;} 2\nmid n\leq d. \]
		Therefore, for any $1\leq n \leq d$, we have 
		\[P_n(x_1,x_2\dots,x_d)=P_n(-x_1,-x_2\dots,-x_d).\]
		Since all $P_n$ can determine the set $\{x_1,\dots, x_d\}$ (with multiplicities) uniquely, we have 
		\[	\{x_1,x_2\dots,x_d\} =\{-x_1,-x_2,\dots,-x_d\}.\] 
		By assumption that $d$ is odd, there exists $1\leq j\leq d$ such that $x_j=-x_j$, i.e., $x_j=0$. Hence there exists $j$ such that $\cos\frac{2k_j\pi}{N}=0$, which implies that $4\mid N$, as we claimed.
		
	\end{proof}

	For $N=p_1^{a_1}p_2^{a_2}\cdots p^{a_s}_s$  with $p_1< p_2<\dots<p_s$ are distinct primes. We have seen that the condition $2d\in \mathbb{N}p_1+\cdots+\mathbb{N}p_s$ is closely related whether $0$ is an eigenvalue or not. We introduce the following notation to study the multiplicity of zero eigenvalue:
	\[ I(0;N)=\Big\{r\in \mathbb{N} \Big| \;\parbox{25em}{there exist $b_1,\cdots,b_s\in \mathbb{N}$ such that $2r=\sum\limits_{l=1}^{s}b_lp_l $ with at least one of $b_l\geq 2$}\Big\}.\]
	Note that $I(0;N)=\mathbb{N}\setminus\{0,1\}$ if $2\mid N$. We will further study this subset in \S \ref{oddNset}.

	From the above proof of Proposition \ref{zero_eigenvalue_critierion}, we obtain the following result.
	\begin{corollary}\label{great_2_implies_N}
		For $d\geq 1$ and $N=p_1^{a_1}p_2^{a_2}\cdots p^{a_s}_s$ with $p_1< p_2<\dots<p_s$ are distinct primes. Assume that $0$ is an eigenvalue of $T^d_N$. If $d\in I(0;N)$, then the multiplicity	\[ m_{T^d_N}(0)\geq c_dN,\]
		for some constant $c_d>0$ which depends only on $d$.
	\end{corollary}
	
	\begin{proof} We divide it into the cases that are mentioned in Proposition \ref{zero_eigenvalue_critierion}.
		
		Case (i): $2\nmid N$.  If there exists $1\leq l_0\leq s$ such that $b_{l_0}\geq 2$, then we have many different ways to represent $0$ as an eigenvalues. We write $b_{l_0}=2+b'$.  For any $0\leq \delta \leq \frac{\pi}{p_{l_0}}$,  we also have 
		\begin{align}
			0=(\frac{1}{2}(b'+\sum\limits_{l=1,l\neq l_0}^{s}b_l ))&(2\cos0)+\sum_{l=1,l\neq l_0}^{s}b_l(2\cos\frac{2\pi}{p_{l}}+2\cos\frac{4\pi}{p_{l}}+\cdots+ 2\cos\frac{(p_l-1)\pi}{p_{l}}) \notag\\
			&+b'(2\cos\frac{2\pi}{p_{l_0}}+2\cos\frac{4\pi}{p_{l_0}}+\cdots+ 2\cos\frac{(p_{l_0}-1)\pi}{p_{l_0}})  +2C_{p_{l_0}}(\delta).
		\end{align}
		There are $\lfloor \frac{N}{2p_{l_0}}\rfloor+1$ different choices of such $\delta$ such that the above equality is a sum of $d$ eigenvalues of $\Z/N\Z$. Note that $p_{l_0}\leq d$.

		Case (ii): $2\mid N$ and $2\mid d$. We write $d=2b$. For any $0\leq \delta_1,\cdots,\delta_b \leq \frac{\pi}{2}$, we also have 
		\begin{align}
			0=\sum_{k=1}^{b}(2\cos\delta_k+2\cos(\pi-\delta_k)).
		\end{align}
		In fact, for each $k$, there are $\lfloor \frac{N}{4}\rfloor+1$ different choices of such $\delta_k$ such that the above equality is a sum of  $d$ eigenvalues of $\Z/N\Z$.

		Case (iii): $2\mid N, 2\nmid d$ and $p_2\leq d$. We have $d-p_2=2c$ for some $c\in \mathbb{N}$.  Therefore $2d=(2c)\times 2+ 2p_2$. For any $0\leq \delta_1,\cdots,\delta_c \leq \frac{\pi}{2}$ and $0\leq \delta\leq \frac{\pi}{p_2}$,  we have 
		\begin{align}
			0=\sum_{k=1}^{c}(2\cos\delta_k+2\cos(\pi-\delta_k))+2\cos\delta+\sum_{j=1}^{\frac{p_{2}-1}{2}}\big(2\cos(\frac{2j\pi}{p_{2}}+\delta)+2\cos(\frac{2j\pi}{p_2}-\delta)  \big ). 
		\end{align}
		
		Case (iv): $4\mid N, 2\nmid d$ and $p_2>d$. We may also write 
		\[
		0=2\cos\frac{\pi}{2}+(\text{a vanishing sum of $d-1$ terms of eigenvalues of } \Z/N\Z).
		\]
		Note that $d-1$ is even, therefore, we have lots of representatives of $0$ in terms of $d$ eigenvalues as given in the case (ii).
		
	\end{proof}
	
	We will see in Proposition \ref{zero_infinity_iff} that the condition $d\in I(0;N)$ is also a necessary condition for higher multiplicity of the zero eigenvalue.

	\subsection{Multiplicity of the eigenvalue zero}
	
	We first recall an elementary result in number theory.
	
	\begin{lemma}\label{lowerbound_pq}
		Let $p<q$ be two distinct odd primes and $d$  a positive integer. If $2d\geq \text{max}\{(p-1)(q-2),p+q+1\}$, then there exist $k_1,k_2\in \mathbb{N}$ such that $2d=k_1p+k_2q$ with at least one of $k_i\geq 2$.
	\end{lemma}
	\begin{proof}
		First we show that if $2d\geq (p-1)(q-2)$, then $2d\in \mathbb{N}p+\mathbb{N}q$. In fact, there exist $A,B\in \mathbb{Z}$ such that 
		\[ Ap+Bq=2d\]
		by the assumption that $p\neq q$. Let $s\in \Z$ such that $0\leq B-sp< p$. We take $k_1=A+sq$ and $k_2=B-sp$. It suffices to show that $$k_1=\frac{2d-k_2q}{p}\geq 0.$$ Since $2d=k_1p+k_2q$, we have $2\mid k_1+k_2$.  This claim follows from the following calculations:
		
		Case 1: if $k_1$ is even, so is $k_2$. Therefore 
		\[ k_1=\frac{2d-k_2q}{p}\geq \frac{(p-1)(q-2)-(p-1)q}{p}=\frac{2-2p}{p}>-2.\]
		
		Case 2: if $k_1$ is odd, then $0\leq k_2\leq p-2$. Hence 
		\[ k_1=\frac{2d-k_2q}{p}\geq \frac{(p-1)(q-2)-(p-2)q}{p}=\frac{q+2-2p}{p}>-1.\]
		
		Now $2d\geq \text{max}\{(p-1)(q-2),p+q+1\}$, it means that we have taken the terms $p, q, p+q$ out of the set $\mathbb{N}p+\mathbb{N}q$; hence the claim of this lemma follows.

	\end{proof}
	
	\begin{remark}
		Note that the bound $(p-1)(q-2)$ is already optimal, since $(p-1)(q-2)-2=-2p+(p-1)q\notin \mathbb{N}p+\mathbb{N}q$.  However, there exist certain terms in $\mathbb{N}p+\mathbb{N}q$ which are less than $(p-1)(q-2)$. For example, $p=7, q=11$, we have $2\cdot 7+2\cdot 11=36<54=(7-1)(11-2)=3\cdot 7+3\cdot 11$ and $4\cdot 7=28<54$.
	\end{remark}

	\begin{proposition}\label{zero_infinity_iff}
		Let $ N=p_1^{a_1}p_2^{a_2}\cdots p^{a_s}_s$ with $p_1< p_2<\dots<p_s$ are distinct primes. Assume that  $0$ is an eigenvalue of $T^d_N$. Then the  multiplicity	\[ m_{T^d_N}(0)\geq c_dN \]
		for some constant $c_d>0$ which depends only on $d$,  if and only if $d\in I(0;N)$.
	\end{proposition}
	\begin{proof}
		Corollary \ref{great_2_implies_N} gives the sufficient part. Now we show the necessary part.  Note that the claim automatically holds if $2\mid N$. Hence in the following we assume that $p_1\geq 3$. We prove this  by contradiction. 	Note that the multiplicity of zero being greater than $c_dN$ implies that,  the number of ways to represent $0$ as a sum of eigenvalues of $\Z/N\Z$ grows as fast as or faster than $N$.  We write
		\begin{equation}\label{eq:0asRU}
			0=\sum_{j=1}^{d}2\cos\alpha_j=\sum_{j=1}^{d}(e^{i\alpha_j}+e^{-i\alpha_j}),
		\end{equation}
		which is a vanishing sum of roots of unity of length $2d$.
		
		If $s=2$, by \cite[Corollary 3.4]{LL},  up to rotations, the only minimal vanishing sums of roots of unity are given by 
		\[ 1+\zeta_{p_1}+\cdots+\zeta_{p_1}^{p_1-1}=0  \mathrm{\;\;and\;\;} 1+\zeta_{p_2}+\cdots+\zeta_{p_2}^{p_2-1}=0 . \]
		Since by the contrary of the condition, the only way to write $2d\in \mathbb{N}p+\mathbb{N}q$ is  $2d=p_1+p_2$. Therefore,
		\begin{equation}
			\sum_{j=1}^{d}(e^{i\alpha_j}+e^{-i\alpha_j})=\eta_1(1+\zeta_{p_1}+\cdots+\zeta_{p_1}^{p_1-1})+\eta_2(1+\zeta_{p_2}+\cdots+\zeta_{p_2}^{p_2 -1}),
		\end{equation}
		for some $N$-th roots of unity $\eta_i$ and we may assume that each $\eta_i$ is not a $p_i$-th root of unity  if $\eta_{i}\neq 1$, for $i=1,2$. Now  applying Lemma \ref{conj_formal_sum}, we obtain that $\eta_1=\eta_2=1$.

		It means that there is only one way to write $0$ as a sum of eigenvalues of $\Z/N\Z$. We get a contradiction in this case. 
		
		For $s\geq 3$, again we  may decompose \eqref{eq:0asRU} into  a sum of minimal vanishing sums of roots of unity:
		\begin{equation}\label{eq:0-sum of minimal}
			\sum_{j=1}^{d}(e^{i\alpha_j}+e^{-i\alpha_j})=M_1+\cdots+M_k .
		\end{equation}
		By \cite[Theorem 4.8]{LL}, up to rotations, minimal sums are either symmetric or asymmetric of length no less than $(p_1-1)(p_2-1)+(p_3-1)$.  If there exists an asymmetric minimal sum in the decomposition \eqref{eq:0-sum of minimal}, then 
		\[ 2d\geq  (p_1-1)(p_2-1)+(p_3-1) \geq \mathrm{max}\{(p_1-1)(p_2-2),p_1+p_2+1\}.\] 
		Then by Lemma \ref{lowerbound_pq}, $2d=b_1p_1+b_2p_2$ and at least one of $b_i\geq 2$. Hence we may assume that each $M_r$ is symmetric. By the contrary of the condition that at least one of $b_l\geq 2$, we write $2d=p_{j_1}+\cdots+p_{j_k}$ and each $M_r$ is symmetric of length $p_{j_r}$. We have an equality of formal sums
		\begin{equation}
			\sum_{j=1}^{d}(e^{i\alpha_j}+e^{-i\alpha_j})=\sum_{r=1}^{k}\eta_r(1+\zeta_{p_{j_r}}+\cdots+\zeta_{p_{j_r}}^{p_{j_r}-1}),
		\end{equation}
		for some  $N$-th roots of unity $\eta_r$ and we assume that each $\eta_r$ is not a $p_{j_r}$-th root of unity if $\eta_{r}\neq 1$,  for $r=1,\cdots, k$. Now applying Lemma \ref{conj_formal_sum} again, we have  $\eta_1=\cdots=\eta_k=1$.
		
		In summary, there are only finitely many ways to write $0$ as a sum of eigenvalues of $\Z/N\Z$, and this contradicts to the assumption again. Hence the claim follows.
		
	\end{proof}

	\subsection{The set $I(0;N)$ for odd $N$}\label{oddNset}
	In this subsection, we give some sufficient conditions to ensure large multiplicity of the zero eigenvalue, which are much easier to check.

	\begin{corollary}\label{odd_infinity_suff}
		Assume that $2\nmid N=p_1^{a_1}p_2^{a_2}\cdots p^{a_s}_s$ with $p_1< p_2<\dots<p_s$ are distinct primes. Assume that $0$ is an eigenvalue of $T^d_N$.  By the proof of Lemma \ref{lowerbound_pq}, we write $(p_1-1)(p_2-2)=k_1p_1+k_2p_2$ for some $k_i\in \mathbb{N}$.   If $d$ satisfies one of the following conditions:
		\begin{itemize}
			\item   $2d\geq \text{max}\{(p_1-1)(p_2-2),p_1+p_2+1\}$;
			\item  $2d=k_ip_i$ for some $1 \leq i\leq s$ and $k_i\in \mathbb{N}$;
			\item $2d=k'_1p_1+k'_2p_2$ with $1\leq k'_i$ and $(k'_1,k'_2)\not=(1,1)$.
			
		\end{itemize}

		then 
		the  multiplicity	\[ m_{T^d_N}(0)\geq c_dN \]
		for some constant $c_d>0$ which depends only on $d$. 
	\end{corollary}
	
	\begin{proof}
		This follows from Lemma \ref{lowerbound_pq},  Proposition \ref{zero_infinity_iff} and the proof of Proposition \ref{zero_eigenvalue_critierion}(i).
	\end{proof}

	\begin{remark}
		In fact, if $s=2$, then the three conditions in Corollary \ref{odd_infinity_suff} together give the set $I(0,N)$. But, for example $s=3$ and $p_1=3, p_2=17, p_3=19$. We see that $(p_1-1)(p_2-2)=30 > 28=3p_1+p_3$. 
	\end{remark}
	
	\begin{remark}
		Assume that $N$ is odd and $15\mid N$. Since $8=3+5$,  $0$ is an eigenvalue of $T^4_N$ with a finite multiplicity. However,  for any eigenvalue $2\cos\alpha$ of $\Z/N\Z$ and $0\leq \delta < \frac{\pi}{3}$,
		\begin{equation*}
			2\cos\alpha=2\cos\alpha+2\cos\delta+2\cos(\frac{2\pi}{3}+\delta)+2\cos(\frac{2\pi}{3}-\delta)
		\end{equation*}
		is an eigenvalue of $T^4_N$ with multiplicity at least $\lfloor \frac{N}{6} \rfloor$.
	\end{remark}

	\section{Multiplicity of a general eigenvalue}
	
	In this section, we show our main result Theorem \ref{all_infinity_iff} and give an upper bound of the multiplicities of all eigenvalues.

	Any eigenvalue $\mu$ of  $T^d_N$ can be written as  $2\sum\limits_{j=1}^{d}\cos\frac{2k_j\pi}{N}$ for some $0\leq k_j \leq \lfloor \frac{N}{2}\rfloor$. For any permutation $\sigma\in S_d$ the symmetry group of $d$ elements, it is clear that  $\mu=2\sum\limits_{j=1}^{d}\cos\frac{2k_{\sigma(j)}\pi}{N}$. 
	
	\begin{lemma}\label{multi_inequality}
		For an eigenvalue $\mu=2\sum\limits_{j=1}^{d}\cos\frac{2k_j\pi}{N}$ of  $T^d_N$ and any $\sigma\in S_d$,  then we have
		\begin{equation*}
			m_{T^d_N}(\mu) \geq m_{T^{d_1}_N}(2\sum_{j=1}^{d_1}\cos\frac{2k_{\sigma(j)}\pi}{N})\cdot m_{T^{d_2}_N}(2\sum_{j=d_1+1}^{d_1+d_2}\cos\frac{2k_{\sigma(j)}\pi}{N}),
		\end{equation*}
		for any integers   $d_1,d_2 \geq 1$ satisfying $d_1+d_2=d$. 
		
		In particular, if $\mu$ is also an eigenvalue of $T^{d-r}_N$ for some positive integer $r$, then there exists another family $\{k_j'\}_{j=1}^{d-r}$ such that $\mu=2\sum\limits_{j=1}^{d-r}\cos\frac{2k_{j}'\pi}{N}$, and it follows that 
		\begin{equation*}
			m_{T^d_N}(\mu) \geq  m_{T^{d-r}_N}(2\sum_{j=1}^{d-r}\cos\frac{2k_{j}'\pi}{N})\cdot m_{T^{r}_N}(0).
		\end{equation*}
		
	\end{lemma}
	\begin{proof}
		The second assertion follows from the first inequality. But for the first claim,  it suffices to show the case that $\sigma=id$. We may also write $\mu_j=2\cos\frac{2k_j\pi}{N}$ for short.  Now $\mu=\sum\limits_{j=1}^{d}\mu_j$,  and for any $1 \leq d_1 \leq d-1$,  any other expressions of  the value of $\sum\limits_{j=1}^{d_1}\mu_j$ or $\sum\limits_{j=d_1}^{d}\mu_j$ will give one  expression of the eigenvalue $\mu$. Therefore, we have the inequality on multiplicities.
	\end{proof}
	
	\begin{definition}
		A sum of $N$-th  roots of unity $\zeta_1+\cdots+\zeta_k$ is called $k$-admissible if $k$ is even and there exist integers $0\leq m_1,\cdots, m_{k/2} \leq \frac{N}{2}$ such that 
		\[ \zeta_1+\cdots+\zeta_k=\sum_{j=1}^{\frac{k}{2}}(e^{\frac{2\pi m_j}{N}i}+e^{-\frac{2\pi m_j}{N}i}).\]
	\end{definition}
	\begin{remark}
		To be $k$-admissible implies that the value of $\zeta_1+\cdots+\zeta_k$ is an eigenvalue of $T^{\frac{k}{2}}_N$, but not the reverse. For example $-1+1+e^{\frac{2\pi}{3}i}+e^{\frac{4\pi}{3}i}$ is not $4$-admissible for $N=6$, but the value $-1=2\cos\pi+2\cos\frac{\pi}{3}$ is an eigenvalue of $T^2_6$. 
	\end{remark}

	\begin{lemma}\label{FMVS_fix_one}
		Given an $N$-th root of unity $\zeta$ and a positive integer $\ell$. The number of choices of minimal vanishing sum of length $\ell$ containing the term $\zeta$ is uniformly bounded by a constant $E_\ell$ which depends only on $\ell$.
	\end{lemma}
	
	\begin{proof}
		The number of such a minimal vanishing sum equals to the number of non-degenerated solutions $(\zeta_1,\dots,\zeta_{\ell-1})$ of 
		\[ -1=\frac{\zeta_1}{\zeta}+\cdots+\frac{\zeta_{\ell-1}}{\zeta}.\]
		By Theorem \ref{Evertsethm} of Evertse, the number of such solutions is bounded by $E_\ell=\ell^{3\ell^2}$.
	\end{proof}

	In the following, we assume $0$ is the only eigenvalue of $T^0_N$ for our convenience.
	\begin{theorem} \label{all_infinity_iff}
		Let  $N=p_1^{a_1}p_2^{a_2}\cdots p^{a_s}_s$ with $p_1< p_2<\dots<p_s$ are distinct primes. Let $d\geq 2$ and $\mu$ be an eigenvalue of $T^d_N$. Then the following are equivalent:
		\begin{itemize}
			\item[(i)] there exists $r\in \mathbb{Z}_{\geq 1}$  such that $\mu$ is an eigenvalue of $T^{d-r}_N$ and $r\in I(0;N)$;
			\item[(ii)]  its multiplicity satisfies
			\begin{equation*}
				m_{T^d_N}(\mu)\geq cN
			\end{equation*}
			for some constant $c>0$ which depends only on $d$. 
		\end{itemize}
	\end{theorem}
	\begin{proof}
		Note that (ii) follows from (i) by Lemma \ref{multi_inequality}.	If $\mu=0$, then we just take $r=d$. Then the equivalence is given by Proposition \ref{zero_infinity_iff}.

		In the following, we assume $\mu\neq 0$ and prove that (ii) implies (i) by contradiction.  Let
		\[ \mathrm{SC}(\mu; N,d)=\Big\{(\alpha_1,\cdots,\alpha_d)\in (\mathbb{Q}\pi)^d \Big| \; \parbox{23em}{there exist $m_1,\cdots, m_d$ such that each $\alpha_{j}=\frac{2 m_j\pi}{N}$, and $\mu=2\sum\limits_{j=1}^{d}\cos \alpha_j$}\Big\}.\]
		
		It is enough to show the cardinality of the set
		\[ \mathrm{SC}(\mu; N,d)\setminus \bigcup_{r\in I(0;N)\cap [1,d]}\mathrm{SC}(\mu; N,d-r)\]
		is uniformly bounded by a constant that depends only on $d$ and does not depend on $N$. On the one hand, We have a natural injective map
		\begin{align*}
			\varphi:	\mathrm{SC}(\mu; N,d)  & \to\mathrm{SR}(\mu;N,2d)\\
			(\alpha_1,\cdots,\alpha_d) &\mapsto (e^{i\alpha_{1}}, e^{-i\alpha_{1}},\cdots, e^{i\alpha_{d}}, e^{-i\alpha_{d}}),
		\end{align*}
		
		where  	\[ \mathrm{SR}(\mu; N,2d)=\Big\{(\zeta_1,\cdots,\zeta_{2d})\in (U_N)^{2d} \Big| \; \mu=\sum\limits_{j=1}^{2d}\zeta_j\Big\}.\]
		On the other hand,  we have 
		\[ \mathrm{SR}(\mu; N,2d)=\bigcup_{k=1}^{2d} \mathrm{SR}^{nd}(\mu; N,2d,k) \]
		with \[ \mathrm{SR}^{nd}(\mu; N,2d,k)=\Big\{ (\zeta_1,\cdots,\zeta_{2d})\in \mathrm{SR}(\mu; N,2d) \Big|\;\parbox{14.1em}{there exists $\sigma\in S_{2d}$ such that $\mu=\sum\limits_{j=1}^{k}\zeta_{\sigma(j)}$ is non-degenerated} \Big\}.\]
		
		It suffices to show that the cardinality of the following set
		\[ W:= \mathrm{SR}^{nd}(\mu; N,2d,k) \cap \varphi \Big( \mathrm{SC}(\mu; N,d)\setminus \bigcup_{r\in I(0;N)\cap [1,d]}\mathrm{SC}(\mu; N,d-r)\Big)\]
		is uniformly bounded by a constant $C_d$ that depends only on $d$. 
		
		Let $\underline{\zeta}=(\zeta_1,\cdots,\zeta_{2d})\in W$. By assumption, there exists $\sigma\in S_{2d}$ such that $\mu=\sum\limits_{j=1}^{k}\zeta_{\sigma(j)}$ is non-degenerated. By Theorem \ref{Evertsethm} of Evertse, we know that the choice of $\zeta_{\sigma(1)},\cdots,\zeta_{\sigma(k)}$ is bounded by a constant $A_d$ that depends only on $d$. We take one, and denote it as $M_0$, which is of length $k$. The rest $\zeta_{\sigma(k+1)}+\cdots+\zeta_{\sigma(2d)}$ is a vanishing sum of length $2d-k$. We say it has a type $(d_1,\cdots,d_{\ell})$-decomposition if  there exist minimal vanishing sums $M_1,\cdots, M_{\ell}$ such that $$\zeta_{\sigma(k+1)}+\cdots+\zeta_{\sigma(2d)}=M_1+\cdots+M_{\ell}$$ and the length $\epsilon(M_j)$ is $d_j$ for any $1\leq j\leq \ell$.  Since $2d-k=\sum\limits_{j=1}^{\ell}d_j$, there are only finite many different types of decomposition. Moreover, the number of types of decomposition is bounded by a constant $B_d$ which depends only on $d$. 
		
		We choose one type $(d_1,\cdots,d_{\ell})$-decomposition and bound the choices of $N$-th roots of unity in each minimal vanishing sums of this fixed type in the following.  We write 
		\[\mu=M_0+M_1+\cdots+M_\ell,\]
		where $M_0$ has been chosen as above and the length of each $M_j$ is $d_j$ has been fixed.
		
		Case 1: If $M_0$ is $k$-admissible, then $\mu$ is an eigenvalue of $T^{\frac{k}{2}}_N$ and hence $0$ is an eigenvalue of $T^{d-\frac{k}{2}}_N$. Since $\underline{\zeta}\in W$, we see that $d-\frac{k}{2}\notin I(0; N)$. Therefore, the number of ways to represent $0$ as eigenvalues of $T^{d-\frac{k}{2}}_N$ is bounded by a constant $D_{d-\frac{k}{2}}\leq D_d$ which depends only on $d$. Hence in this case, the possibility of choice of $\underline{\zeta}$ is bounded by $(2d)!A_dB_dD_d$.
		
		Case 2: If $M_0$ is not $k$-admissible and let $\zeta$ be one of the unpaired roots, then $\overline{\zeta}$ must appear in one of $M_1,\cdots,M_\ell$. Without loss of generality, we assume it is in  $M_1$. By Lemma \ref{FMVS_fix_one}, the choice of $M_1$ is bounded by a constant $E_{d_1}\leq E_d$ that depends on $d$. We just take one, and consider $M_0+M_1$.
		
		Now we do induction:  Let $M_0, M_0+M_1,\cdots, M_0+M_1+\cdots+M_i$ be not $k, k+d_1,\cdots, k+d_1+\cdots+d_i$ admissible respectively, with the conjugate of one of unpaired roots in $M_0+\cdots+M_j$ lies in $M_0+\cdots+M_j+M_{j+1}$, for $1\leq j\leq i-1$. If  $M_0+M_1+\cdots+M_i+M_{i+1}$ is $k+d_1+\cdots+d_{i+1}$-admissible, then as in Case 1, the number of choices of $M_{i+2}+\cdots+M_{\ell}$ is bounded by $D_d$. Otherwise, as in Case 2 and Lemma \ref{FMVS_fix_one}, we have $E_{d_{i+1}}\leq E_d$ choices of $M_{i+1}$.

		Since $M_0+M_1+\cdots+M_\ell \in W$ is $2d$-admissible, the above process stops at most after $\ell$ steps. Hence the possibility of choices of $\underline{\zeta}$ in $W$ is bounded by 
		\begin{align*}
			C_d:&=(2d)!A_dB_d(D_d+E_d(D_d+E_d(D_d+\cdots)))\\
			&=(2d)!(D_d+E_dD_d+E_d^2D_d+\cdots+E^{\ell-1}_dD_d+E_d^{\ell}).
		\end{align*}
		
	\end{proof}

	As we see that there are many eigenvalues with large multiplicities in the higher dimensional case, however, we always have a finiteness result for a range of eigenvalues by  Theorem \ref{all_infinity_iff} above:
	\begin{corollary}\label{rang_bound}
		For $N\geq 3$, we list all eigenvalues of $T^d_N$ as $\mu_1\geq \mu_2 \geq \cdots \geq \mu_{N^d}$. Then there exists a constant $c=c_d$, $0< c \leq \frac{1}{2}$, such that if $k\in [1, cN^d) \cup ((1-c)N^d, N^d] $,  $N\geq 3$,  then the multiplicity $m_{T^d_N}(\mu_k)$ is uniformly bounded by a constant $C_d$ which only depends on $d$. 	
	\end{corollary}
	
	\begin{proof}
		Let $\mu$ be an eigenvalue of $T^d_N$. We first claim that if $2(d-2)<\mu \leq 2d $ or $-2(d-2)>\mu \geq -2d $, then the  multiplicity $m_{T^d_N}(\mu)$ is uniformly bounded by a constant which only depends on $d$. Otherwise, by Theorem \ref{all_infinity_iff},  there exists $r\geq 2$ such that $\mu$ is an eigenvalue of $T^{d-r}_N$. But the range of $\mu$ is $-2(d-2) \leq -2(d-r)\leq \mu \leq 2(d-r)\leq 2(d-2)$. This contradicts to the assumption that $\mu>2(d-2)$ or $\mu < -2(d-2)$.  Hence the claim follows. 
		
		Let $\alpha = \frac{\cos^{-1}(1-\frac{2}{d})}{2\pi }$. Then for any $t$, $0\leq t < \alpha N$,  we have $\cos \frac{2t \pi}{N} > 1- \frac{2}{d}$.  Hence,  any eigenvalue of the form 
		\[\mu = 2\cos (\frac{2t_1 \pi}{N}) + 2\cos (\frac{2t_2 \pi}{N}) +\cdots +  2\cos (\frac{2t_d \pi}{N}) \]
		with $0\leq t_j <\alpha N$, $1\leq j \leq d$, is bigger than $2(d-2)$.  There are at least $(\alpha N-1)^d$ such eigenvalues of $T^d_N$.  For safety, we take $c_d= \frac{1}{2}\alpha^d$. Then, for any $k \in [1,  c_d N^d)$, the multiplicity $m_{T^d_N}(\mu_k)$ is uniformly bounded by a constant $C_d$ which only depends on $d$. By symmetry, the conclusion also holds for $k \in ((1-c_d) N^d, N^d]$.
	\end{proof}
	
	\begin{remark}
		As we mentioned before in Introduction, the constant $c_d$ here is clearly not optimal. 
	\end{remark}
	Another consequence is the following upper bound for all eigenvalues:

	\begin{corollary}\label{general_bound}
		Let $\mu$ be an eigenvalue of $T^d_N$ and $p_1$ the minimal prime factor of $N$.  Then $m_{T^d_N}(\mu)\leq C_dN^{\frac{d}{p_1}} $, where $C_d$ is a constant, which depends only on $d$.
	\end{corollary}
	\begin{proof}
		We have seen that the minimal length of a vanishing sum of  cosines is $p_1$, hence, there are at most $\lfloor  \frac{d}{p_1}\rfloor$ in any sum of term $d$ cosines. 
	\end{proof}
	\begin{remark}\label{optimal_bound_all}
		The exponent $\frac{d}{p_1}$ in Corollary \ref{general_bound} is optimal. For example, if $d=kp_1$, we have \[ 0=\sum_{j=1}^{k}\left( 2\cos\delta_j+\sum_{m=1}^{\frac{p_1-1}{2}}2\huge(\cos(\frac{2m\pi}{p}+\delta_j)+\cos(\frac{2m\pi}{p}-\delta_j)\huge)\right),\]
		for any $0 \leq \delta_j \leq \frac{N}{p_1}$. It implies that the multiplicity of zero in $T_N^{kp_1}$ is at least $(\frac{N}{p_1})^k$.
	\end{remark}

	\section{Eigenvalue multiplicity of abelian Cayley graphs}
	Although in this article we only consider the Cayley graph $(\Z/N\Z)^d$ with the generating set $S= \{\pm e_1,\dots, \pm e_d\}$. The method used here apparently works for arbitrary abelian Cayley graphs. Recall that for any abelian Cayley graph $(G, S)$,  $G$ an abelian group and $S$ a symmetric generating set of $G$, its eigenvalues are 
	given as  
	\[ \mu=\sum_{g \in S} \chi(g),\]
	where $\chi$ is any character of $G$. See Babai \cite{Babai}. Then $\mu $ is a sum of cosines since $S$ is symmetric. To determine the multiplicity of $\mu$, one studies
	vanishing sums of cosines as those in this paper. To characterize eigenvalues with high multiplicities, however, is more subtle. For example,  we have the following new phenomenon. 
	
	\begin{example}
		Let $(\Z/N\Z,S)$ be the Cayley graph associated with the cyclic group $\Z/N\Z$ with generating set $S=\{ \pm 1, \pm (\frac{N}{p}+1), \pm (\frac{2N}{p}+1), \dots, \pm (\frac{(p-1)N}{p}+1)\}$ for a prime $p\mid N$. Then, for any $0\leq k\leq N-1$ and $p\nmid k$, we have eigenvalues
		\[ \mu_k=2\cos(\frac{2k\pi}{N})+2\cos((\frac{N}{p}+1)\frac{2k\pi}{N})+2\cos((\frac{2N}{p}+1)\frac{2k\pi}{N})+\dots + 2\cos((\frac{(p-1)N}{p}+1)\frac{2k\pi}{N})=0.\]
		This means that the zero-eigenvalue has a very high multiplicity.  In particular, if $N$ is even, and we take $p=2$, the generating set $S=\{\pm 1, \pm(\frac{N}{2}+1)\}$ has only four elements. However, for any odd $k$,
		\[ \mu_k=2\cos \frac{2k\pi}{N}+2\cos ((\frac{N}{2}+1)\frac{2k\pi}{N})=0.\]
		Then the multiplicity of zero eigenvalue is at least $\frac{N}{2}$ in this case.
	\end{example}

	The above example shows that large eigenvalue multiplicities can also be caused by relations given by the generating set. Determining the eigenvalue multiplicity of a general abelian Cayley graph with an arbitrary generating set is an interesting problem and we hope to come back to this topic in future.
	
	\vskip 6mm
	\noindent 
	{\bf Acknowledgements.} 
	We would like to thank Prof.\,Wei Zhu for the help of some numerical verifications and Prof.\,Jordan Ellenberg,  Prof.\,Mounir Hajli and Prof.\,Yichao Zhang for helpful discussions; especially, Prof.\,Ze\'ev Rudnick for his many valuable suggestions, which helped to greatly improve the readability of this article.

\end{document}